\documentclass[pdflatex,sn-mathphys]{sn-jnl}

\jyear{2021}%

\theoremstyle{thmstyleone}%
\newtheorem{thm}{Theorem}[section]
\newtheorem{lem}[thm]{Lemma}
\newtheorem{prop}[thm]{Proposition}
\newtheorem{defi}[thm]{Definition}
\newtheorem{cor}[thm]{Corollary}

\theoremstyle{thmstyletwo}%
\newtheorem{remark}[thm]{Remark}%

\begin{document}

\title[The Stable Picard Group of $A(n)$]{The Stable Picard Group of $A(n)$}


\author[1]{\fnm{JianZhong} \sur{Pan}}\email{pjz@amss.ac.cn}
\equalcont{These authors contributed equally to this work.}
\author*[2]{\fnm{RuJia} \sur{Yan}}\email{yanrujia@amss.ac.cn}
\equalcont{These authors contributed equally to this work.}

\affil[1]{\orgdiv{HUA Loo-Keng Key Laboratory of Mathematics, Academy of Mathematics and Systems Science}, \orgname{Chinese Academy of Sciences}, \orgaddress{\street{No.55, Zhongguancun East Road}, \city{Beijing}, \postcode{100190}, \country{China}}}
\affil*[2]{\orgdiv{Academy of Mathematics and Systems Science}, \orgname{Chinese Academy of Sciences}, \orgaddress{\street{No.55, Zhongguancun East Road}, \city{Beijing}, \postcode{100190}, \country{China}}}


\abstract{ In this paper, we showed that the Stable Picard group of $A(n)$ for $n\geq 2$ is $\mathbb{Z}\oplus \mathbb{Z}$ by considering the endotrivial modules over $A(n)$. The proof relies on reductions from a Hopf algebra to its proper Hopf subalgebras.}




\maketitle

\section{Introduction}\label{sec1}

Let $H$ be a cocommutative connected finite graded Hopf algebra over a base field $k$ with characteristic $p$. The Picard group of $Stab(H)$, denoted by $Pic(H)$ is the group of graded stably $\otimes$-invertible $H$-modules,
$$Pic(A):=\{M\in Stab(H):\exists N\text{ such that } M\otimes N=\underline{k}\},$$
where $\underline{k}$ is the unit of the symmetric monoidal category $(Stab(A),\otimes, \underline{k})$. In this paper, we are interested in the determination of the Picard group of $A(n)$, the Hopf subalgebra of $\mathcal{A}_2$ generated by $Sq^1,Sq^2,\cdots,Sq^{2^n}$ (see section \ref{Steenrod algebra}). The situation when $n=1$ was calculated by Adams and Priddy in \cite{BSO} to show the uniqueness of $BSO$ in 1976. It says that $$Pic(A(1))=\mathbb{Z}\oplus\mathbb{Z}\oplus \mathbb{Z}/2.$$
The $\mathbb{C}$-motivic version of the above result was given by Gheorghe, Isaksen and Ricka in \cite{gheorghe2018picard}, which is
$$Pic(A_{\mathbb{C}}(1))=\mathbb{Z}\oplus\mathbb{Z}\oplus\mathbb{Z}\oplus\mathbb{Z} .$$
In 2017, Bhattacharya and Ricka  determined in \cite{bhattacharya2017stable} that 
$$Pic(A(2))=\mathbb{Z}\oplus\mathbb{Z}.$$
The authors of \cite{bhattacharya2017stable} conjectured that $$Pic(A(n))=\mathbb{Z}\oplus\mathbb{Z}$$ 
for all $n>2$, which is exactly the main result of this paper (section \ref{A(n)}). 

The ideal is to consider $\mathbb{T}(H)$, the group of endotrivial modules over $H$ (See section \ref{defi endotrivial}).  An $H$-module is called endotrivial if $End_k(M)=k\oplus(free)$. By definition if $M$ is endotrivial, then $M$ is invertible since $End_k(M)=M^*\otimes M$. On the other hand, since $H$ is cocommutative, the stable module category $Stab(H)$  is a stable homotopy category (section  9.6 of \cite{hovey1997axiomatic}),  and hence a closed symmetric monoidal category. Therefore Theorem A.2.8 of \cite{margolis2011spectra}  guarantees that if $M$ is invertible, then it is endotrivial. That is, the group  of endotrivial modules of $H$  is isomorphic to the stable Picard group of  $H$. 

The word $endotrivial$ comes from the representation theory of finite groups, since the group algebra $kG$ of a finite $p$-group $G$ is a well-defined Hopf algebra with diagonal coproduct. The group of endotrivial modules over $kG$ has been completely figured out by joint efforts of many mathematicians including Puig, Alperin, Dade, Carlson and Thevenaz, check \cite{carlson2000torsion} for references.

 The main object of this paper is to give a proof of the following main result:
\begin{thm}
[Corollary \ref{main result}] \label{what we want}
	Given $n\geq 2$, the morphism of groups
		$$f:\mathbb{Z}\oplus \mathbb{Z}\to Pic(A(n))$$
	which sends $(m,l)$ to  $[\sigma(m)\Omega_{A(n)}^l(\mathbb{Z}/2)]$ is an isomorphism of groups.
\end{thm}

This article is organized as follows. Section 2 introduces some preliminaries and conventions. In section \ref{main}, we showed that the problem of classifying endotrivial $H$-modules can somehow be reduced to the one about each $E$-modules for certain Hopf subalgebras $E$ (See Theorem \ref{3.2 steenrod re}). With help of this reduction theorem, the Picard group of $A(n)$ was given in section \ref{A(n)}.

Throughout this paper, $k$ will denote a field with positive character $p$.
Every algebraic structure is implicitly over the base field $k$, and tensor products are
taken over $k$. The Hopf algebras under consideration in this paper are cocommutative,  unless explicitly specified otherwise.

\section{Preliminaries}

A Hopf algebra is a bialgebra over field $k$ (of characteristic $p>0$) with a compatible antipode map $S$ (\cite{montgomery1993hopf}). We will adapt the following notation: If $H$ is a Hopf algebra, denote $\triangledown:H\otimes H\to H$ as its product; $\triangle:H\to H\otimes H$ as its coproduct; $\epsilon:H\to k$ and $\eta:k\to H$ are counit and unit respectively.
	
Suppose $M,N$ are $H$-modules, then $M\otimes N$ is an $H$-module defined by $h(m\otimes n)=\sum_i a_im\otimes b_in$, where $h\in H$, $m\in M$ and $n\in N$ with $\triangle (h)=\sum_i a_i\otimes b_i$.  For an $H$-module $M$, the dual $M^*$ is defined  by $M^*=Hom_k(H,k)$, with the natural $H$-module structure 
$$(h\cdot f)(v)=f(S(h)v) \quad \forall h,v\in H;\ f\in M^*$$

\subsection{Finite connected graded Hopf algebra}

If the underlying algebra of $H$ is decomposed into a direct sum of vector spaces:
$$H=\bigoplus_{n=0}^{\infty}H_n=H_0\oplus H_1 \oplus H_2\oplus \cdots$$
such that $H_iH_j\subset H_{i+j}$ for all nonnegative $i$ and $j$, we say that $H$ is a graded algebra.
		
		Moreover, a graded $H$-module $M$ is an $H$-module which has a decomposition as direct sum of vector spaces:
		$$M=\bigoplus_{-\infty}^{\infty}M_n$$
		such that $H_i\cdot M_j\subset M_{i+j}$. When $M$, $N$ are graded $H$-module, both $M\otimes N$ and $M^*$ defined above are still graded modules. In the rest of the paper, all algebras and modules are assumed to be graded.
		
		An algebra (or module) is finite if it is a  finite dimensional vector space over the base field.
		
		Let $grH$ be the category with objects finitely generated graded  $H$-modules, and morphisms graded  $H$-homomorphisms. The $i$ th shift functor, $\sigma(i):grH\to grH$, is defined by $\sigma(i)X=\sigma(i)(\oplus X_n)=Y=\oplus Y_n$, where $Y_n=X_{n-i}$, and the obvious way on morphisms.\bigskip

 Suppose $H$ is a finite graded Hopf algebra over a field $k$, then there is a minimal $n\in \mathbb{Z}^+$ such that $H_j=0$ if $j>n$, and we say that the top degree of $H$ is $n$ and denote as $\vert H\vert =n$. 
 
 The trivial $H$-module $k$  can be viewed as a graded $H$-module $\underline{k}$ where $\underline{k}_0=k, \underline{k}_j=0$ for $j\neq 0$. We still use $k$ to denote $\underline{k}$ if there is no confusion.
 
 Recall that a graded algebra $\Lambda$ is connected if $\Lambda_0=k$, and we have  
  \begin{defi}
 	A finite graded connected  algebra $\Lambda$ is a Poincare algebra if there is a map of graded $k$-modules $e:\Lambda\to \sigma(n)k$ for some $n$ such that the pairing $\Lambda_q\otimes \Lambda_{n-q}\xrightarrow{\triangledown}\Lambda_n\xrightarrow{e}k$ is non-singular. 
 \end{defi}
 In particular, this implies that if $\Lambda$ is a Poincare algebra, then $\Lambda_i=0$ for $i>n$  and $\Lambda_n\cong k$.
\begin{thm}[Theorem 12.2.5 \& 12.2.9 of \cite{margolis2011spectra}]
	If $H$ is a finite graded connected Hopf algebra, then $H$ is a Poincare algebra.
\end{thm}
 So if $H$ is a finite graded connected Hopf algebra and $\vert H\vert =n$, $H$ is a Poincare algebra and thus $H\cong \sigma(n) H^*$. Namely, $H_0\cong H_n\cong k$.
 \begin{thm}[Proposition 12.2.8 of \cite{margolis2011spectra}]\label{proj=free}
 	Let $\Lambda$ be a Poincare algebra and let $M$ be a $\Lambda$-module. Then the followings are equivalent:
 	
 	(1) $M$ is free
 	
 	(2) $M$ is projective
 	
 	(3) $M$ is flat 
 	
 	(4) $M$ is injective.
 \end{thm}

 Another fact worth to be known is that a Hopf algebra of our interest is always free over its Hopf subalgebras:

 \begin{thm}[Theorem 1.3 of \cite{aguiar2013lagrange}]\label{graded lagrange rewrite}
	Let $H$ be a finite graded connected Hopf algebra. If $Z\subset H$ is a Hopf subalgebra, then $H$ is a free left (and right) $Z$-module. Moreover,
	$$H\cong Z\otimes (H/Z^+H)$$
	as left $Z$-modules and as graded vector spaces.
\end{thm}

\begin{remark}
 The proof of \cite{aguiar2013lagrange} applies for arbitrary characteristic.
\end{remark}

In the rest of this article, all Hopf algebras are assumed to be finite graded connected Hopf algebras.
 
 \subsection{Endotrivial modules} \label{defi endotrivial}

 An $H$-module $M$ is called endotrivial if there is an isomorphism of $H$-modules $End_k(M)\cong \underline{k}\oplus F$, where $F$ is a free $H$-module. As an $H$-module,  $End_k(M)\cong M\otimes M^*$, thus $M$ is endotrivial iff $M\otimes M^*\cong k\oplus F$ for some free module $F$. Note that $M\otimes N$ and $M^*$ are endotrivial if $M$ and $N$ are endotrivial.
 
 	Two endotrivial $H$-modules $M$ and $N$ are said to be equivalent if 
 $$either \quad M\cong N\oplus (free) \quad or \quad N\cong M\oplus (free).$$
 
 Denote the equivalent class of an endotrivial module $M$ as $[M]$, and define $$\mathbb{T}(H)=[M]\vert \text{ M  is an endotrivial  H-module}.$$
 $\mathbb{T}(H)$ is an abelian group with unit $k$, and $[M]+[N]=[M\otimes N]$.
 
 \begin{lem}
 	If an endotrivial $H$-module $M$  decomposes as $M=M_1\oplus M_2$, then one of the summands is free and the other one is endotrivial.
 \end{lem}
 \begin{proof}
   By Theorem \ref{proj=free}, $H$ is indecomposable as an $H$-module.  By definition,	
   $$k\oplus F\cong End_k(M_1)\oplus End_k(M_2)\oplus Hom_k(M_1,M_2)\oplus Hom_k(M_2,M_1).$$ 
  Krull-Schmidt Theorem shows that all the summands on the right hand side are free except one. Denote it as $L$, one has  $dimL=1$ mod $dimH$. On the other hand,  $dim(Hom_k(M_1,M_2))=dim (Hom_k(M_2,M_1))=dimM_1dimM_2$, thus $L\neq Hom_k(M_2,M_1)$ and $L\neq Hom_k(M_1,M_2)$. Without loss of generality, we assume $L=End_k(M_1)=k\oplus (free)$ and $End_k(M_2)=(free)$. Therefore $End_k(M_2)\otimes M_2=M_2\otimes M_2^*\otimes M_2$, is free and has a direct summand $M_2$. By Theorem \ref{proj=free}, $M_2$ is free.
 \end{proof}

 Therefore,	in the equivalence class $[M]$, there is, up to isomorphism, a unique indecomposable module $M_0$ and every module in the class is isomorphic to $M_0\oplus (free)$.
 
 \begin{defi}[Definition 2.15 of \cite{gheorghe2018picard}]
 Consider the projective resolution of $k$:
 $$ker(\epsilon)\to H\to k.$$
 The operator $\Omega_H$ is given by
$$\Omega_H M=ker(\epsilon)\otimes M,$$
where $\epsilon:H\to k$ is the augmentation of $H$. For $m\geq 0$, define $\Omega^m_H M$ inductively to be $\Omega_H(\Omega_H^{m-1}M)$. For $m < 0$, define $\Omega_H^m M$ to be $(\Omega_H^{-m} (M^*))^*$.
 \end{defi}

As showed in section 2.3 of \cite{gheorghe2018picard}, $\Omega_H^m(k)$ is endotrivial and different choices of resolution will give the same equivalence class in $\mathbb{T}(H)$,
so the choice we made is harmless.
 
 When $H=kG$, the group algebra of a finite $p$-group $G$, we know well about endotrivial modules on $kG$ (\cite{carlson2000torsion}).  With  help of this example, we get:
 
 \begin{prop}[Prop3.2 of \cite{bhattacharya2017stable}]\label{end of induction}
 	Given an elementray Hopf algebra $B$ over $k$ generated by $\{x_1,x_2, \cdots, x_t\}$  where $t\geq 2$. Then
 	$f:\mathbb{Z}\oplus \mathbb{Z}\to\mathbb{T}(B)$ with $f(m,n)=[\sigma(m)\Omega_{B}^n(k)]$ is an isomorphism of groups.
 \end{prop}
\begin{remark}
 A Hopf algebra $H$ over $k$ is called elementary if it is bicommunicative and has $x^p=0$ for all $x\in ker \epsilon$.
\end{remark} 
 \begin{remark}
  The proof of \cite{bhattacharya2017stable} still works for $p>2$.
 \end{remark}
 
 \subsection{Mod 2 Steenrod algebra}\label{Steenrod algebra}
 We now turn  to the algebra of central interest to us, the mod $2$ Steenrod algebra $\mathcal{A}_2$ and its subalgebra $A(n)$. The background knowledge about $\mathcal{A}_2$ in this section  mainly comes from 15.1 of \cite{margolis2011spectra}.\bigskip
 
 The mod $2$ Steenrod algebra $\mathcal{A}_2$ is a graded vector space over $\mathbb{Z}/2$ with basis all formal symbols $Sq(r_1,r_2,\cdots)$, called Milnor basis, where $r_i\geq 0$ and $r_i > 0$ only finitely often. If $r_j=0$ for $j>i$ we will also write $Sq(r_1,\cdots, r_i)$ and $\vert Sq(r_1,r_2,\cdots)\vert$=$\sum(2^i-1)r_i$. 
 The product of $\mathcal{A}_2 $ are defined as follows:
 $$
 Sq(r_1,r_2,\cdots)\cdot Sq(s_1,s_2,\cdots)=\sum_{X}\beta(X)Sq(t_1,t_2,\cdots)
 $$
 the summation being over all matrices 
 $$X=	
 \left[	
 \begin{matrix}	
 	*	    & x_{01} &x_{02}  &\cdots \\	
 	x_{10}	& x_{11} & \cdots & 	   \\
 	x_{20}  & \vdots &        &       \\	
 	\vdots	&  	     &        &  
 	
 \end{matrix}	 
 \right]	
 $$
 with nonnegative $x_{ij}$ satisfying $\sum_{i}x_{ij}=s_j$ and $\sum_{j}2^jx_{ij}=r_i$, and we will call such matrix allowable.
 
 Then take $t_k=\sum_{i+j=k}x_{ij}$ and $\beta(X)=\prod_{k}(x_{k0},x_{(k-1)1},\cdots x_{0k})\in\mathbb{Z}/2$, where $(n_1,\cdots,n_r)=(n_1+\cdots n_r)!/n_1!\cdots n_r!$  reduced mod $2$.
 In particular, there is always an allowable matrix with $x_{i0}=r_i$ and $x_{0j}=s_j$ and all the other entries are zero, we call such allowable matrix trivial.

 The coproduct is defined as follow:
 $$
 \triangle Sq(r_1,r_2,\cdots) = \sum_{r_i=s_i+t_i} Sq(s_1,s_2,\cdots) \otimes Sq(t_1,t_2,\cdots ).
 $$
 
 Here is a useful Lemma for the calculation of $(n_1,\cdots,n_r)$.
 \begin{lem}\label{dyadic}
 	Suppose  the unique dyadic expansion of the natural number $n$ is  $n=\sum_{j}a_j2^j$, then we say  $2^i\in n$ if $a_i=1$; and $2^i\notin n$ if $a_i=0$.\\ 
 	Then 
 	$(n_1,\cdots,n_r)=1$ if and only if $2^i\in n_j$ implies $2^i\notin n_k$ for $k\neq j$. 
 \end{lem}
 
 We denote the Milnor basis element $Sq(0,\cdots,0,r) $ with $r$ at the $t$-th position as $P_t(r)$, and denote $P^s_t=P_t(2^s)$. The homogeneous primitives of $\mathcal{A}_2$ are precisely the $P^0_t$ for all $t\geq 1$. 
 
 \begin{lem}\label{s<t s=t}Here are some direct consequences about the product in $\mathcal{A}_2$:
 	
 	(1) If $r<2^v$ and $u<2^t$ then $[P_t(r),P_v(u)]=0$.
 	
 	(2) If $r<2^t$ then $P_t(r)P_t(s)=(r,s)P_t(r+s)$, in particular $P_t(r)^2=0$. 
 	
 	(3) $P^t_t\cdot P^t_t= P_t(2^t-1)P^0_{2t}$.
 	
 	(4) If $1 \leq r < 2^t$ then $[P^t_t, P_t(r)]=P_t(r-1)P^0_{2t}$
 	
 	(5) If $1\leq r\leq 2^i\leq 2^t$ and $r<2^t$, then $[P_i(r),P^i_t]=P_i(r-1)P_{t+i}^0$
 	
 	(6) If $\leq 2^i\leq r\leq 2^t$, then $[P_i^0,P_t(r)]=P_t(r-2^i)P_{t+i}^0$
 \end{lem}
 \begin{proof}
 	(1)-(4) comes from Lemma 1.3 of \cite{adams1992modules}.
 	
 	By definition, if $X=(x_{ij})$ is an allowable matrix for the product $Sq(r_1,r_2\cdots)\cdot Sq(s_1,s_2\cdots )$, then the entries satisfies 
 	\begin{gather*}
 		\sum_{i}x_{ij}=s_j\\
 		\sum_{j}2^jx_{ij}=r_i.
 	\end{gather*}
 	As for (5) $P_i(r)P^i_t $ only has trivial allowable matrix, while beside the trivial matrix, $P^i_tP_i(r)$ has another allowable matrix $X=(x_{uv})$ where $x_{0i}=r-1$, $x_{ti}=1$ and all the other entries are zero.  Note that $\beta(X)=1$, therefore $[P^{i}_{t},P_{i}(r)]=Sq(r_1,r_2\cdots )$ with $r_{i}=r-1$, $r_{t+i}=1$ and all the other coordinates are zero. On the other hand $P_i(r-1)P_{t+i}^0$ only has trivial allowable matrix, thus $P_i(r-1)P_{t+i}^0=Sq(r_1, r_2\cdots )=[P^{i}_{t},P_{i}(r)]$.

 	As for (6), $P^0_iP_t(r)$ only has trivial allowable matrix, while beside the trivial matrix, $P_t(r)P^0_i$ has another allowable matrix $X=(x_{uv})$ where $x_{t0}=r-2^i$, $x_{ti}=1$ and all the other entries are zero.  Note that $\beta(X)=1$, therefore $[P^0_i,P_t(r)]=Sq(r_1,r_2\cdots )$ with $r_{t}=r-2^i$, $r_{t+i}=1$ and all the other components are zero. Meanwhile, $P_t(r-2^i)P_{t+i}^0=Sq(r_1,r_2\cdots)$with $r_{t}=r-2^i$, $r_{t+i}=1$ and all the other components are zero, since the product only has trivial allowable matrix, we are done.
 \end{proof}
 
 \begin{lem}\label{fen jie}
 	Suppose $r=\sum_{i=0}^t \alpha_i 2^i$ is the dyadic decomposition of $r $ and $r<2^{t+1}$, then $P_t(r)=\prod_{\alpha_{i}=1}P^i_t$.
 \end{lem}
 \begin{proof}
 	The lemma follows from (2) of Lemma \ref{s<t s=t} and Lemma \ref{dyadic}.
 \end{proof}
 
 We define the excess of a Milnor basis element by 
 $$ex(Sq(r_1,r_2,\cdots))=r_1+r_2+\cdots.$$
 This relates well to the product:
 \begin{lem}\label{excess}[Lemma 15.1.2 of \cite{margolis2011spectra}]
 	$$Sq(r_1,r_2,\cdots\cdot) Sq(s_1,s_2,\cdots )=bSq(r_1+s_1,r_2+s_2,\cdots )+\Sigma Sq(t_1,t_2,\cdots )$$ with $b=\prod(r_i,s_i)$ and $ex( Sq(t_1,t_2,\cdots ))<ex(Sq(r_1+s_1,r_2+s_2,\cdots ))$
 \end{lem}

 \begin{defi}
 	Let $B$ be a Hopf subalgebra of $\mathcal{A}_2$ and define the profile function of $B$,  $h_B:\{1,2,\cdots\}\to \{0,1,\cdots, \infty\} $ by $$h_B(t)=\min\{s\vert r_t<2^s\ for\  all\ Sq(r_1,r_2,\cdots) \in B\},$$ and $h_B(t)=\infty $ if no such $s$ exists. 
 \end{defi}
 
 The importance of profile functions and the $P^s_t$'s can  be seen in the following classification theorem:
 \begin{thm}[Theorem 15.1.6 of \cite{margolis2011spectra}]\label{class of sub hopf alg}
 	Given a  Hopf subalgebra $B$ of $\mathcal{A}_2$, then 
 	
 	(a) $B$ is spanned by the Milnor basis elements in it; precisely, $B$ has a $\mathbb{Z}/2$-basis $\{Sq(r_1,r_2,\cdots)\vert r_t<2^{h_B(t)}\}$.
 	
 	(b) $B$ is generated as an algebra by $\{P^s_t\vert s<h_B(t)\}$.
 	
 	(c) $h:\{1,2,\cdots\}\to \{0,1,\cdots, \infty\} $ is the profile function of a   Hopf subalgebra if and only if for all $u,v\geq 1$, $h(u)\leq v+h(u+v)$ or $h(v)\leq h(u+v)$. Moreover, the algebra being normal if the latter condition is always satisfied.
 \end{thm} 
\begin{remark}
 Suppose $h_A(t)$ and $h_B(t)$ are two profile functions. Then $A$ is a Hopf subalgebra of $B$ if and only if $h_A(t)\leq h_B(t)$ for all $t$.
\end{remark}
 
\begin{remark}
Suppose $h_A(t)$ and $h_B(t)$ are two profile functions for graded Hopf subalgebras. Then $m(t)=min\{h_A(t),h_B(t)\}$ is the profile function of $A\cap B$. 
\end{remark}
 
 \begin{defi}Here are two examples:
 	
 	(1) For each $t$ define a Hopf subalgebra $J(t)$ by the profile function $h(u)=0, u\neq t$ and $h(t)=t$. Then $J(t)$ is an exterior algebra on generators $P^0_t,\cdots, P^{t-1}_t$. 
 	
 	(2) $A(n)$ is defined by the profile function $h(t)=max\{n+2-t,0\}$. It has a minimal generating set $\{P^0_1,P^1_1,\cdots P^n_1\}.$
 \end{defi}
 
\begin{remark}
 $A(n)$ is the Hopf subalgebra generated by $Sq^1,Sq^2,\cdots ,Sq^{2^n}$.
\end{remark}
 
 \begin{cor}\label{chengji feiling}
 	Suppose $B$ is a finite Hopf subalgebra of $\mathcal{A}_2$, then 
 	$$\prod_{t=1}\left(\prod_{s=0}^{h_{B}(t)-1}P^{s}_{t}\right)\neq 0.$$
 \end{cor}
 \begin{proof}
 	Iterated applications of Lemma \ref{excess} tell that
 	$$\prod_{t=1}\left(\prod_{s=0}^{h_{B}(t)-1}P^{s}_{t}\right)=Sq(r_1,r_2,\cdots) + (\text{terms with smaller excess}) \neq 0$$
 	where $r_i=2^{h_B(i)}-1$. 
 \end{proof}
 Next, we introduce conditions for a function to be a profile function.
 \begin{lem}\label{lemma 1}
 	If $h_2$ is a profile function, $h_1:\{1,2,\cdots \}\to \{0,1,2\cdots \}$ is a function such that 
 	
 	(1) $h_1(t)\leq h_2(t)$, for all $t\geq 1$.
 	
 	(2)  $\exists t_0\geq 1$ such that $h_1(t)=h_2(t)$ for $t\geq t_0$ .
 	
 	Then $h_1$ is a profile function if $h_1(u)\leq v+h_1(u+v)$ or $h_1(v)\leq h_1(u+v)$, for all $u,v\geq 1$ and $u+v\leq t_0$.
 \end{lem}
 \begin{proof}
 	If $u+v> t_0$ and $h_1(v)>h_1(u+v)=h_2(u+v)$ then $h_2(v)\geq h_1(v)>h_2(u+v)$. By assumption, $h_2$ is a profile function, thus $h_2(u)\leq v+h_2(u+v)$. Therefore $h_1(u)\leq h_2(u)\leq v+h_2(u+v)=v+h_1(u+v)$. 
 \end{proof}
 
 \begin{lem}\label{lemma 2}
 	Suppose $h$ is a profile function, $h_1:\{1,2,\cdots \}\to \{0,1,2\cdots \}$ is a function such that  $$h_1(t)=\begin{cases}
 		h(t)& t\neq t_0\\
 		0 & t=t_0
 	\end{cases}$$
 	Then $h_1(t)$ is a profile function if for $1\leq u< t_0$,   $h_1(t_0-u)>0$ implies $h_1(u)\leq t_0-u$.
 \end{lem}  
 \begin{proof}
 	Since $h_1(t_0)=0$, $h_1(t_0)\leq h_1(t_0+u)\leq u+h_1(t_0+u)$.
 	
 	By assumption, for $1\leq u< t_0$,   $h_1(t_0-u)>0$ implies $h_1(u)\leq t_0-u$. Therefore,  $h_1(u)\leq t_0-u+h_1(t_0) $ or $h_1(t_0-u)\leq h_1(t_0)$.
 	
 	For the rest situation, namely $1\leq u,v$ such that $u\neq t_0$, $v\neq t_0$ and $u+v\neq t_0$, one has $h_1(u)=h(u)$, $h_1(v)=h(v)$ and $h_1(u+v)=h(u+v)$. Since $h$ is a profile function, $h_1(u)\leq v+h_1(u+v)$ or $h_1(v)\leq h_1(u+v)$, for all such $u,v$. 
 \end{proof}

Suppose $A\subset \mathcal{A}_2$ is a Hopf subalgebra and $B\subset A$  is normal Hopf subalgebra of $A$. In order to describe $A//B$, it's helpful to know more about $AB^+=B^+A$, the ideal generated by $B$.

 \begin{lem}\label{AB^+}
 	Suppose $A\subset \mathcal{A}_2$ is a Hopf subalgebra and $B=C\cap A$  where $C$ is normal Hopf subalgebra of $\mathcal{A}_2$. If $\forall s, t$ such that $0\leq s< h_C(t)$, one has $2^s\notin r_t$, then $Sq(r_1,r_2,\cdots)$ is not a summand of any
 	$x \in AB^+=B^+A$.
 \end{lem}
 \begin{proof}
 	By the proof of Theorem 15.1.6 (c) of \cite{margolis2011spectra}, we have that
 	$$\mathcal{A}_2C^+=C^+\mathcal{A}_2=span\{Sq(u_1,u_2\cdots) \vert  \exists 0\leq s< h_C(t)\text { such that }  2^s\in u_t\}.$$ 
 	Thus by assumption, $Sq(r_1,r_2,\cdots)$ is not a summand of any $x\in \mathcal{A}_2C^+=C^+\mathcal{A}_2$ and thus $Sq(r_1,r_2,\cdots)$ is not a summand of any 
 	$x\in  AB^+=B^+A$.
 \end{proof}
 
 We state the following two lemmas here for later argument. 
 
 \begin{lem}\label{class of elementary }
 	Given a finite  Hopf subalgebra $A$ of $\mathcal{A}_2$. If $B$ is a maximal elementary Hopf subalgebra of $A$ then $B=B_i\cap A$ for some $i$, where 
 	$$h_{B_i}(t)=\begin{cases}
 		0& t<i\\
 		i& t\geq i
 	\end{cases}$$
 \end{lem}
 \begin{proof}
  By example A.6. of \cite{palmieri1997note}, $B_i$ are all the maximal elementary Hopf subalgebras of $\mathcal{A}_2$. If $B\subset A\subset \mathcal{A}_2$ is an elementary subHopf algebra, then $B\subset B_i\cap A$ for some $i$. On the other hand, $B_i\cap A$ themselves are elementary Hopf subalgebra, thus if $B$ is maximal, then $B=B_i\cap A$ for some $i$.
 \end{proof}
\begin{remark}
 The maximal elementary Hopf subalgebras of $A(n)$ are exactly $\{B_i\cap A(n)\}$ where $1\leq i\leq [n/2]+1$, since $B_i\cap A(n) \subset B_{[n/2]+1}\cap A(n)$ if $i\geq [n/2]+1$ and $B_i\cap A(n)\nsubseteq B_j\cap A(n)$ if $i,j\leq [n/2]+1$ and $i\neq j$.
\end{remark}

 Let $E(n)\subset A(n)$ be the sub algebra  generated by $\{P^0_t\vert  t\leq n+1\}$. Then $E(n)$ is an exterior algebra and is a normal Hopf subalgebra of $A(n)$.
 
 \begin{lem}[Proposition 15.3.29 of \cite{margolis2011spectra}]\label{A//E}
 	There is a map $\eta:A(n-1)\to A(n)//E(n)$ which doubles degree and which is an algebra isomorphism.
 \end{lem}
In general a map (resp. isomorphism) of algebras induces a functor (resp. isomorphism) of the module categories. Here $\eta$ introduces a small technicality.
\begin{prop}[Proposition 15.3.30 of \cite{margolis2011spectra}]\label{A//E 2}
	$\eta$ induces an isomorphism $$\eta^*: gr(A(n)//E(n))^{ev}\to grA(n-1)$$
\end{prop}
 Here $gr(A(n)//E(n))^{ev}$ denotes the full subcategory of $gr(A(n)//E(n))$ of modules concentrated in even degree.

 \section{Reduction to Hopf subalgebras}\label{main}
 
In this section we hope to reduce the calculation of $\mathbb{T}(H)$ to the calculation of $\mathbb{T}(E)$, where $E$ is some proper Hopf subalgebra of $H$ (Theorem \ref{3.2 steenrod re}). Then in some sense we can do an induction (see next section). This section is devoted to prove Theorem \ref{3.2 steenrod re}. \bigskip
 
 Recall that by assumption $H$ is a Poincare algebra and thus $H\cong \sigma(n) H^*$. Namely, $H_0\cong H_n\cong k$.  We denote the unique basis (up to a nonzero scale multiplication) of the $1$-dimensional vector space $H_n$ as $t^H_1$.
 
 \begin{defi}
 	Let $M$ be an $H$-module and $Z$ be a normal Hopf subalgebra of $H$. Define
 	$$M^Z:=\{x\in M\vert  hx=\epsilon(h)x, \forall h\in Z\}, \quad M^Z_1:=\{t^Z_1\cdot m\vert  \forall m\in M\}.$$ 
 	$M^Z$ is called the (left) invariant $Z$-submodule of $M$.
 \end{defi}
\begin{remark}
 $M^Z$ is a well defined $H//Z$-module where
$\alpha\cdot m= h\cdot m,\quad \forall \alpha\in H//Z, m\in M^Z,$ 
and $h$ is an arbitrary preimage of $\alpha $ of the quotient map $\pi:H\to H//Z$.
\end{remark}

It's easy to verify that $Hom_H(k,M)\cong \{x\in M\vert  hx=\epsilon(h)x, \forall h\in H\}$ as graded vector spaces. Notice that $H$ itself is an $H$-module and we can also consider the invariant $H$-submodule of $H$.
 
 \begin{lem}
 Suppose $\vert H\vert =n$, then $H^H=H_n$
 \end{lem}
 \begin{proof}
 	Let $x=\sum_{i=0}^{n}x_i\in H^H$, with $x_i\in H_i$.  And let $x_j$ be the first nonzero summand. Then since $H$ is a Poincare algebra, there is $y_{n-j}\in H_{n-j}$ such that $y_{n-j}x_j=t^H_1\neq0$. Therefore, $y_{n-j}x=t^H_1$ for grading reason. On the other hand, since $x\in H^H$, $y_{n-j}x=\epsilon(y_{n-j})x$, which leads to $\epsilon(y_{n-j})x=t^H_1$. This is true only if $x=at^H_1$ and $y_0=1/a$ for some nonzero $a\in k$. Thus  $H^H\subset H_n$. 
 	By definition, $\forall h\in H^+$, $\vert ht^H_1\vert >n$, hence $ht^H_1=0=\epsilon(h)t^h_1$. Therefore,  $H_n\subset H^H$. 
 \end{proof} 
 
 By the same argument, $(\sigma(m)H)^H=\sigma(m)H_n$, the top degree  of $\sigma(m)H$. Another observation is that $(\sigma(m)H)^H_1=\sigma(m)H_n$ , since $t^H_1\cdot \sigma(m)H_i\subset \sigma(m)H_{n+i}=0$ for $i>0$ and $t^H_1\sigma(m)H_0=\sigma(m)H_n$.
 
 In general, $M^H_1$ and $M^H$ are not equal. However, this is true when $M$ is free.
 \begin{lem}\label{4.6 rewrite}
 	When $M$ is a free $H$ module,  $M^H=M^H_1$.
 \end{lem}
 \begin{proof}
 	Suppose $\vert H\vert =n$.
 	$\forall h\in H,  m\in M$, $h\cdot t^H_1m=\epsilon(h)t^H_1m$ by definition, thus $M^H_1\subset M^H$.
 	
 	On the other hand, suppose $\{m_1,\cdots m_s\}$ is an $H$-basis for $M$ as free module, then $\{t^H_1(m_1),\cdots,t^H_1(m_s)\}$ is a basis for $M^H_1$ as vector space, since  
 	$t^H_1\sigma(m)H=\sigma(m)H_n$
 	$M^H\cong Hom_H(k,M)\cong Hom_H(k,\oplus _{i=1}^s \sigma(\alpha_i)H)\cong \oplus_{i=1}^s (\sigma (\alpha_i)H)^H\cong \oplus_{i=1}^s \sigma(\alpha_i)H_n=  \oplus_{i=1}^s \sigma(\alpha_{i+n})k.$
 	Thus $M^H=M^H_1$ since they have same dimension as vector spaces over $k$.
 \end{proof}

The element $t^H_1$ can be used to detect the freeness of  an $H$-module. 
\begin{lem}[Lemma 12.2.6 of \cite{margolis2011spectra}] \label{pre4.9 re}
	If $M$ is an  $H$-module, then multiplication by $t^H_1$ induces a map $f:k\otimes_HM\to M$ (of degree $\vert t^H_1\vert $). Furthermore, $f$ is a monomorphism if and only if $M$ is free.  
\end{lem}

\begin{lem}\label{4.9 re}
	Let $M$ be an $H$-module. Suppose $ t^H_1m_1,t^H_1m_2, \cdots,t^H_1m_n$  are  $k$-linearly independent. Then $m_1,m_2,\cdots, m_n$ generate a free  $H$-submodule $L\subset M$.
\end{lem}
\begin{proof}
	By assumption,  $k\otimes_H L$ is a $k$ vector space spanned by $1\otimes m_1,1\otimes m_2,\cdots, 1\otimes m_n$. Consider the map $f:k\otimes_HL\to L$ defined in Lemma\ref{pre4.9 re}:
	$$f(\sum_{i=1}^{n}a_i\otimes m_i)=\sum_{i=1}^{n}a_it^H_1m_i$$ 
	where $a_i\in k$. Since $ t^H_1m_1,\cdots, t^H_1m_n$  are $k$-linearly independent, $f$ is injective.
	
 By Lemma \ref{pre4.9 re}, we are done.
\end{proof}

In fact,  $t^H_1$ is just the integral of the Hopf algebra $H$ (\cite{montgomery1993hopf}), and it has some kind of transitivity:
 \begin{lem}\label{4.7 write}
 	Suppose  $Z \subset H$ is a normal Hopf subalgebra. Denote the generator  of $Z_{\vert Z\vert }$ and $(H//Z)_{\vert H//Z\vert }$ as $t^Z_1$ and $\overline {t^H_Z}$ respectively. Then $t^H_Zt^Z_1$ is the generator of $H_{\vert H\vert }$ where $t^H_Z$ is a preimage of  $\overline {t^H_Z}$ under the quotient map.
 \end{lem}
 \begin{proof}
 	Show that $t^H_Zt^Z_1$ is well-defined first:
 	$$\forall a \in HZ^+,\ \ (t^H_Z+a)t^Z_1=t^H_Zt^Z_1+ hzt^Z_1=t^H_Zt^Z_1$$  
 	where $a=hz$ for some $h\in H$ and $z\in Z^+$.
 	
 	Recall that $H$ is a Poincare algebra. By definition, $\exists x\in H_{\vert H\vert -\vert Z\vert }$ such that $xt^Z_1\neq 0$. By Theorem \ref{graded lagrange rewrite}, $x\in H_{\vert H//Z\vert }$. Since $H_{\vert H\vert }=k$, is a one dimensional vector space generated by $t^H_1$, we can assume that $$xt^Z_1=t^H_1$$ after a scalar multiplication if necessary. 
 	
 	Now that $xt^Z_1\neq 0$, one has $x\notin HZ^+$ and hence the quotient image of $x$ in $H//Z$ is nonzero. Therefore $x$ is a preimage of $\overline{t^H_Z}$ since $\vert x\vert =\vert H//Z\vert $ and $(H//Z)_{\vert H//Z\vert }=k$, a one dimensional vector space.
 \end{proof}    

Thanks to the transitivity, we have the following two technical lemmas: 
\begin{lem}\label{4.8 rerite}
	Let $Z \subset H$ be a normal Hopf subalgebra. Suppose $M$ is a finitely generated $H$-module such that $M^Z_1$ is a free  $H//Z$ module, and $M\downarrow^H_Z\cong k\oplus (free)$. Then $M\cong k\oplus (free)$ as $H$-modules.
\end{lem}                                                   \begin{proof}
	Let $t^Z_1m_1,\cdots, t^Z_1m_n$ be a basis of $M^Z_1$ as free $H//Z$-module. Then by Lemma \ref{4.7 write} and Lemma \ref{4.6 rewrite}, we have 
	$$t^H_Zt^Z_1m_1=t^H_1 m_1,\cdots, t^H_Zt^Z_1m_n=t^H_1 m_n$$ 
	 are $k$-linearly independent. Therefore $\{m_i\}$ generate a free $H$-submodule of $M$ by Lemma \ref{4.9 re}, named as $L$. Therefore, $L$ is also a free $Z$-module, which implies $L^Z=L^Z_1=M^Z_1$, generated by $t^Z_1m_1,\cdots, t^Z_1m_n$ as free $H//Z$-module.
	On the other hand, $M^Z_1=F^Z_1=F^Z$, where by assumption   $M\downarrow^H_Z\cong k\oplus F$ for some $F$ free as $Z$-module. Therefore $L^Z=F^Z$, hence $dim(L)=dim(F)$ since  both $L$ and $F$ are free $Z$-modules. Thus $dim(L)=dim(M)-1$ and since $L$ is free, hence injective, as an $H$-module, we get $M\cong L\oplus k$, as was to be shown.   
\end{proof}  

\begin{cor}\label{4.11 re}
	Let $Z$ be a normal  Hopf subalgebra. Then $$H^Z\cong \sigma(\vert Z\vert ) H//Z$$ as $H//Z$-modules.
\end{cor}
\begin{proof}
	By Theorem \ref{graded lagrange rewrite}, and Lemma \ref{4.6 rewrite}, as graded $k$-vector spaces, 
	$$H^Z=t^Z_1\cdot Z\otimes H//Z=k\{t^Z_1\}\otimes H//Z\cong \sigma(\vert Z\vert )H//Z.$$
	
	On the other hand, 
	$$\overline{t^H_Z}H^Z=t^H_Zt^H_1H=t^H_1H=k\{t^H_1\}.$$
	Then by Lemma \ref{4.9 re}, $H^Z$ has a free $H//Z$-submodule. Therefore, $H^Z\cong \sigma(\vert Z\vert ) H//Z$ as $H//Z$- modules for dimensional reason.
\end{proof}

\begin{defi}
	Suppose $A$ is an algebra and $M$ is an $A$-module. Let $\mathcal{B}=\{B_i\vert i\in I\}$ be a collection of nontrivial subalgebras of $A$.  If $M$ is free iff $M$ restricted  to every subalgebra $B\in \mathcal{B}$ is free, then we say $A$ has detect property in $grA$, with detecting set $\mathcal{B}$.
\end{defi}
 Here $B$ is nontrivial means $B\neq A$ and $B\neq k$. 

\begin{thm}[Theorem 1.3, Example A.6 of \cite{palmieri1997note}]\label{steenrod alegbra   detect projective re}
	Suppose $E$ is a finite Hopf subalgebra of the mod $2$ Steenrod algebra $\mathcal{A}_2$, $M$ is an $E$-module. Then $M$ is projective if and only if $M$ restricted to $B$ is projective for every elementary Hopf subalgebra $B$ of $E$.
\end{thm}
\begin{remark}
Here $E$ and $B$ are all graded, and the grading was given by that of $\mathcal{A}_2$. 
\end{remark}
\begin{remark}
 In particular,  if $E$ is not elementary, $E$ has detect property in $grE$ with detecting set $\{B\vert B\text{ is a maximal  elementary Hopf subalgebra } \}$.
\end{remark}

Now we are ready to prove our main theorem:
\begin{thm}\label{3.2 steenrod re}
	Suppose $H$ has a nontrivial normal Hopf subalgebra $Z$ such that $H//Z$ has detect property in $gr(H//Z)$ with detecting set $\mathcal{B}=\{B_i\vert i\in I\}$. Then  
	$$Res: \mathbb{T}(H)\longrightarrow \prod_{E\in \mathcal{E}}\mathbb{T}(E). $$
	is injective,
	where $\mathcal{E}=\{E_i\vert E_i//Z=B_i ,i\in I\}$.
\end{thm}
\begin{proof}
	By assumption there is a nontrivial normal Hopf subalgebra $Z$ of $H$ such that $H//Z$ has detect property with detecting set $\mathcal{B}$.
	
	Let $M$ be an endotrivial $H$-module whose class in $\mathbb{T}(H)$ is in the kernel of the restriction map $Res$. This means $M\cong k\oplus F$ as $E$-modules, for $E\in \mathcal{E}$, where $F$ is a free $E$-module. Since $Z$ is finite nontrivial connected  Hopf algebra, $\vert t^Z_1\vert \neq 0$ and thus $t^Z_1\in Z^+$. If follows that 	
	$$M^Z=k\oplus F^Z, \quad M^Z_1=F^Z_1.$$
	
	Since $F$ is a free $E$-module, $F$ is also a free $Z$ module, thus by Lemma \ref{4.6 rewrite},  $F^Z_1=F^Z$ . Moreover, $F^Z$ is free over $E//Z$ since $F$ is free over $E$, see Corollary \ref{4.11 re}.
	
	Therefore, $M^Z_1$ is a free $E//Z$ module, for every   $E\in \mathcal{E}$. Using the fact that $H//Z$ has detect property with detecting set $\mathcal{B}=\{B_i\vert i\in I\}=\{E//Z\vert E\in \mathcal{E}\}$, we have $M^Z_1$ is a free $H//Z$-module. By Lemma \ref{4.8 rerite}, we are done. 
\end{proof}

\section{Endotrivial group of A(n)} \label{A(n)}

Consider the  Hopf algebra $A(n), n\geq 2$. From now on, let $a=[(n-1)/2]$, $i$ be an integer with $1\leq i \leq a+1$ and we will work on the field $\mathbb{Z}/2$.

By Lemma \ref{A//E}, there is an isomorphism of algebras $\eta: A(n-1)\to A(n)//E(n)$ which doubles the grading. Use the same symbol as in the proof of  Lemma \ref{class of elementary }, the maximal elementary Hopf subalgebras of $A(n-1)$ are exactly  $B_1\cap A(n-1),\cdots , B_{a+1}\cap A(n-1)$, with which $A(n-1)$  has detect property.  Define $$h_{B_i'(n)}(t)=\begin{cases}
	h_{B_i\cap A(n-1)}(t)+1& 1\leq t\leq n+1 \\
	0& t> n+1
\end{cases}$$
By Lemma \ref{lemma 1}, $B_i'(n)$ is a Hopf subalgebra, since $h_{B_i'(n)}(t)$ is an increasing function when $t\leq i$ and $h_{B_i'(n)}(t)=h_{A(n)}(t)$ when $t\geq i$.

Again by Lemma \ref{lemma 1}, for each $B_i'(n)$ there are  Hopf subalgebras $D_{i}(n)$ defined by 
$$h_{D_{i}(n)}(t)=\begin{cases}
	0& t< i\\
	h_{B_i'(n)}(t)& t\geq i
\end{cases}.$$

Given $n\geq 2$, recall that $B_1'(n)=D_1(n)$. 

There are injective homomorphisms between the group of endotrivial modules for these Hopf subalgebras:

$\mathbb{T}(A(n))\to \prod_{i=1}^{1+a}\mathbb{T}(B_i'(n)) $ (Proposition \ref{step 0})

$\mathbb{T}(B_i'(n))\to \mathbb{T}(\ast)\times \mathbb{T}(D_{i}(n)), 2\leq i \leq a+1$ (Proposition \ref{step 1}).

$\mathbb{T}(D_{i}(n))\to \mathbb{T}(\ast)\times \mathbb{T}(\ast), 1\leq i \leq a+1 $
(Proposition \ref{step 2} \ref{step 3})

 All the above maps are  the restriction map induced by inclusion of Hopf subalgebras, $\ast$ means certain elementary Hopf subalgebras. 

Finally,  we define $O_i=<P^0_t\vert i\leq t\leq n+1>$.

The rest of this section will be devoted to the proof of Theorem \ref{what we want}.
\bigskip

\begin{prop} \label{step 0}
$\forall n\geq 2$,
$$Res: \mathbb{T}(A(n))\longrightarrow \prod_{i=1}^{a+1}\mathbb{T}(B'_i(n))$$ is injective.
\end{prop}

\begin{proof}
	Given a finitely generated   $A(n)//E(n)$-module $M$, assume that $M\vert _{B_i'(n)//E(n)}$ is free in $gr\left(B_i'(n)//E(n)\right)$ for each $1\leq i\leq a+1$, we will show that $M$ is a free $A(n)//E(n)$-module. 
	
	The trick is that  there is a natural splitting $M=M_1\oplus \sigma(1)M_2$ where $M_1$ and $M_2$ are all concentrated in even degree.   This is because that $A(n)//E(n)$ is concentrated in even degree.  Proposition \ref{A//E 2} guarantees that  $(\eta^*M_j)\vert _{B_i\cap A(n-1)}$ is free for $j=1,2$; and then the detect property of $A(n-1)$ with detecting set $\{B_i\cap A(n-1)\vert 1\leq i\leq a+1\}$ implies $\eta^*M_j$ is a free $A(n-1)$-module. Again, since $\eta^*$  is an isomorphism between the two module categories, $M_j$ is a free $A(n)//E(n)$-module, $j=1,2$. Namely, $M=M_1\oplus \sigma(1)M_2$ is a free $A(n)//E(n)$-module.
	
	The above argument shows that $A(n)//E(n)$ has detect property with detecting sets $\{B_i'(n)//E(n)\vert  1\leq i\leq a+1\}$. By Theorem \ref{3.2 steenrod re}, we are done.
\end{proof}

\begin{lem}\label{degree comparison}
	Given two  Hopf subalgebras $E_1$, $E_2$  of $H$ and suppose $M$ is a finitely generated $H$-module such that $M\vert _{E_j}$ is a free $E_j$-module for $j=1,2$.
	
	M is a free H-module, if 

(1)$\exists t_1\neq 0 \in E_1, t_2\neq 0 \in E_2$ such that $\vert t_1t_2 \vert = \vert H\vert$, and

(2)for all homogeneous $y\neq 0  \in E^+_2 , \vert y\vert  > \vert t_1\vert $  or $y=z$ with $z^2 = 0, 0\neq  zt_1 \in E^+_1$.

\end{lem}
\begin{proof}
	Suppose $f:\oplus_{j=1}^qH\to M$ with $f(h_1,h_2,\cdots, h_q)=\sum_{j=1}^{q}h_jm_j$ is a surjective $H$-module homomorphism with $q$ is minimal. Such $f$ and $q$ exist since $M$ is finitely generated. Furthermore, we assume that $\vert m_j\vert \leq \vert m_{j+1}\vert $.  We show that $m_1$ generates a free $H$-submodule first. 
	
	Suppose $m_1=\sum_{l}x_lb_l$ for some homogeneous	$ x_l\in E_1^+, b_l\in M$. Then for each $l$, one has $\vert x_lb_l\vert =\vert x_l\vert +\vert b_l\vert > \vert m_1\vert $, contradiction. Hence $xw$ is not a summand of $m_1$, $\forall x\in E_1^+, w\in M$. By assumption $M$ is a free $E_1$ module, thus $xm_1\neq0$. In particular, $t_1m_1\neq0$. 
	
Suppose $t_1m_1=\sum_{l}y_lc_l$ for some homogeneous $ y_l\in E_2^+, c_l\in M$ and $\exists y_{l_0}\neq z$. Then  $\vert y_{l_0}c_{l_0}\vert =\vert y_{l_0}\vert +\vert c_{l_0}\vert >\vert t_1\vert +\vert c_{l_0}\vert \geq \vert t_1m_1\vert $, which  is impossible. Thus $t_1m_1=z\cdot c$ where $c=\sum_l c_l$, then $z\cdot t_1m_1=(z)^2c=0$. On the other hand, by condition (2), $0\neq zt_1\in E_1$, which is a contradiction since we showed $xm_1\neq 0$, for all nonzero $x\in E_1$. In conclusion, $yw$ is not a summand of $t_1m_1$ for all $y\in E_2^+$ and $w\in M$. Again since $M$ is a free $E_2$ module, $yt_1m_1\neq0$.
	
	In summary, $yt_1m_1\neq0, \forall y\in E_2^+$. In particular, $t_2t_1m_1\neq0$, which implies that $t_2t_1\neq 0$. By assumption, $\vert t_2t_1\vert =\vert H\vert $.
	Therefore, $t_2t_1 =t^H_1$ since $H$ is a finite connected   Hopf algebra and $t_2t_1 \neq 0$.
	By Lemma \ref{4.9 re}, $m_1$ generates a free $H$-submodule $L_1\subset M$.

 	Since $q$ is minimal, $m_j\notin L_1$ for $j\geq 2$.
	Define $M_2=M/L_1$, then $f_2:\oplus_{j=1}^{q-1}H\to M_2$ with $f(h_1,\cdots, h_{q-1})=\sum_{j=1}^{q-1}h_j\overline{m_{j+1}}$ is a surjective $H$-module homomorphism with $q-1$ is minimal
	, where $\vert \overline{m_{j}}\vert \leq \vert \overline{m_{j+1}}\vert $. By an induction on the number of generators, one may assume that $M_2$ is free. Therefore, $M=L_1\oplus M_2$ is a free   $H$-module. 
\end{proof}
The proof of Proposition \ref{step 1}, Proposition \ref{step 2}, and Proposition \ref{step 3} rely heavily on the above Lemma. To apply this Lemma, one must know well about the degree of the related algebraic generators $P^s_t$.

\begin{lem}\label{jisuan}
	For $s,u>0$ and $t,v\geq 1$, 
	
	(1) If $s+t=u+v$ and $t\geq v$ then $\vert P^s_t\vert \geq \vert P^{u}_v\vert $;
	
	(2) If $s+t>u+v$, then $\vert P^s_t\vert >\vert P^u_v\vert $.
\end{lem}
\begin{proof}
	(1) By definition,  $\vert P^s_t\vert =2^s(2^t-1)=2^{s+t}-2^s$ and $\vert P^u_v\vert =2^u(2^v-1)=2^{u+v}-2^u$. By assumption $u\geq s$ and hence $\vert P^s_t\vert \geq \vert P^{u}_v\vert $.
	
	(2) By assumption, $s+t>1$.  By (1), since $P^0_{s+t}=2^{s+t}-1> 2^{s+t-1} > 2^{u+v-1}=P^{u+v-1}_1 $, we are done.
\end{proof}

\begin{prop}\label{step 1}
For each $2\leq i \leq a+1$, $n\geq 2$,
	$$Res: \mathbb{T}(B_i'(n))\longrightarrow \mathbb{T}(E(n))\times \mathbb{T}(D_{i}(n))$$ is injective.
\end{prop}  
\begin{proof}
By Lemma \ref{s<t s=t}, $P^0_{n+1}$ commutes with all the $P^s_t\in B_i'(n)$, thus $<P^0_{n+1}>$ is a normal Hopf subalgebra and $B_i'(n)//<P^0_{n+1}>$ is a connected   Hopf algebra, denote it as $H$. Denote  $E(n)//<P^0_{n+1}>$ and $D_{i}(n)//<P^0_{n+1}>$ as $E_1$ and $E_2$ respectively, then $E_1$ and $E_2$ are connected   Hopf subalgebra of $H$. We will show that $H$ has detect property in $grH$ with detecting set $\{E_1,E_2\}$. 

As an algebra, $E_1$ is generated by $\mathcal{F}=\{\overline{P^0_{t}}\vert  1 \leq t\leq n\}$ , $E_2$ is  generated by $\mathcal{G}=\{\overline{P^{s}_{t}}\vert  i\leq t\leq n, 0\leq s\leq h_{D_i(n)}(t)\}$ and $H$ is generated by 
$\mathcal{F}\cup \mathcal{G}$.
 
 Consider the element  $t_1=\prod_{j=1}^{i-1}\overline{P^0_j}$ and $t_2=\prod_{t=i}^{n}\left(\prod_{s=0}^{h_{B_i'(n)}(t)-1}\overline{P^{s}_{t}}\right)$ where $t_1\in E_1$ and $t_2\in E_2$. By Corollary \ref{chengji feiling} and Lemma \ref{AB^+}, they are nonzero. Next, we show they satisfy the assumption of Lemma \ref{degree comparison}.

 Since $\vert t_1\vert =2^i-i$ and $\vert \overline{P^{0}_{i}}\vert =2^i-1$, by Lemma \ref{jisuan}, $\forall y\in E^+_2$, $\vert y\vert >\vert t_1\vert $.
 
 Notice that the Milnor basis of highest degree in $B_i'(n)$ is $Sq(r_1,r_2,\cdots, r_{n+1})$ with $r_t=2^{h_{B_i'(n)}(t)}-1$, one has 
 $$\vert B_i'(n)\vert =\sum_{t=1}^{n+1}\left(2^{h_{B_i'(n)}(t)}-1\right)\left(2^t-1\right).$$
 Since $\vert P^0_{n+1}\vert =2^{n+1}-1$ and $h_{B_i'(n)}(n+1)=1$, by Theorem \ref{graded lagrange rewrite},
 $$\vert H\vert =\vert B_i'(n)\vert -\vert <P^0_{n+1}>\vert =\sum_{t=1}^{n}\left(2^{h_{B_i'(n)}(t)}-1\right)\left(2^t-1\right).$$
 On the other hand 
 $$\vert t_1\vert +\vert t_2\vert =\sum_{t=1}^n\left(2^{h_{B_i'(n)}(t)}-1\right)\left(2^t-1\right),$$
 Therefore, we have $\vert H\vert =\vert t_1\vert +\vert t_2\vert $. By Lemma \ref{degree comparison}(1), $H$ has detect property in $grH$ with detecting set $\{E_1,E_2\}$. By Theorem \ref{3.2 steenrod re}, we are done.
\end{proof}

\begin{prop} \label{step 2}
	For each $2\leq i \leq a+1$, $n\geq 2$
	$$Res: \mathbb{T}(D_{i}(n))\longrightarrow \mathbb{T}(B_i\cap D_{i}(n))\times \mathbb{T}(B_{i+1}\cap D_{i}(n))$$ is injective.
\end{prop}
 
\begin{proof}
Consider $Y_i(n)\subset D_i(n)$ with profile function
	$$h_{Y_i(n)}(t)=\begin{cases}
		2i+1-t & i+1\leq t\leq 2i-1\\
		0 & otherwise
	\end{cases}.$$

Since $n\geq 2$ and $i\leq a+1$, $Y_i(n)\subset A(2i-1)$ is a Hopf subalgebra of $D_i(n)$ by Lemma \ref{lemma 1} and Lemma \ref{lemma 2}. 
Define $X_i(n)\subset D_i(n)$ by the function 
	$$h_{X_i(n)}(t)=\begin{cases}
	i+1 & t=i\\
	1 & t= 2i\\
	0 & otherwise
\end{cases}.$$
Directly by Theorem \ref{class of sub hopf alg}, $X_i(n)$ is a Hopf subalgebra of $D_i(n)$. 

By Lemma \ref{s<t s=t} (3) $[x,y]=0$, $\forall x\in D_i(n), y\in Y_i(n)$. Namely, $Y_i(n)$ is a normal Hopf subalgebra of $D_i(n)$, and $D_i(n)//Y_i(n)$ is a connected   Hopf algebra, denoted as $H$. Define $E_1'=\left(B_i\cap D_i(n)\right)//Y_i(n)$  and $E_2=\left(B_{i+1}\cap D_i(n)\right)//Y_i(n)$. We will show that $H$ has detect property in $grH$ with detecting sets $\{E_1',E_2\}$.

Suppose $M$ is a finitely generated   $H$-module, and $M\vert _{E_1'}$,  $M\vert _{E_2}$ are free. We want to show that $M$ is a free $H$-module.

By Lemma \ref{AB^+}, $X_i(n)\cap \left(Y_i(n)^+D_i(n)\right)=0$, namely $\pi\vert _{X_i(n)}$ is an isomorphism of graded algebras from $X_i(n)$ to $\pi X_i(n)$, where $\pi :D_i(n)\to H$ is the quotient map. By Lemma \ref{steenrod alegbra   detect projective re} and Lemma \ref{class of elementary }, $X_i(n)$ has detect property in $grX_i(n)$ with detecting set $\{B_i\cap X_i(n)\}$, hence $\pi X_i(n)$ has detect property in $gr\left(\pi X_i(n)\right)$ with detecting set $\{\pi \left(B_i\cap X_i(n)\right)\}$. Now that $M$ is a   free $E_1'$-module and $\pi (B_i\cap X_i(n))$ is a Hopf subalgebra of $E_1'$, it is free over $\pi (B_i\cap X_i(n))$ , and thus $M$ is a free   $\pi X_i(n)$-module. Now, denote $E_1=\pi X_i(n)$.

Consider the element  $t_1=\prod_{j=0}^{i}\overline{P^j_i}$ and $t_2=\prod_{t=i+1}^{n+1}\left(\prod_{s=h_{Y_i(n)}(t)}^{h_{D_i(n)}(t)-1}\overline{P^{s}_{t}}\right)$ where $t_1\in E_1$ and $t_2\in E_2$. By Corollary \ref{chengji feiling} and Lemma \ref{AB^+}, they are nonzero. Next, we show they satisfy the assumption (2) of Lemma \ref{degree comparison}.

Take $z=\overline{P^{0}_{2i}}$, we know $z^2=0$. By definition of the profile functions, $z\in E_1 \cap E_2$. Since $\vert t_1\vert =(2^{i+1}-1)(2^i-1)$ and $\vert \overline{P^{i}_{i+1}}\vert =(2^i-1)2^i$,  by Lemma \ref{jisuan}, $\forall y\in E^+_2$, $\vert y\vert >\vert t_1\vert $ as long as $y\neq z$. Moreover, $zt_1\in E_1$ because $t_1,z\in E_1$ and by Corollary \ref{chengji feiling} and \ref{AB^+}, $zt_1\neq 0$.

By definition, 
$$\vert t_1t_2\vert =\sum_{t=i}^{n+1}\left(2^{h_{D_i(n)}(t)}-2^{h_{Y_i(n)}(t)}\right)\left(2^t-1\right).$$
By the definition of profile  function, 
$$\vert D_i(n)\vert =\sum_{t=i}^{n+1}\left(2^{h_{D_i(n)}(t)}-1\right)\left(2^t-1\right).$$
and 
$$\vert Y_i(n)\vert =\sum_{t=i}^{n+1}\left(2^{h_{Y_i(n)}(t)}-1\right)\left(2^t-1\right).$$
Therefore, we have 
$$\vert H\vert =\vert D_i(n)\vert -\vert Y_i(n)\vert =\vert t_1\vert +\vert t_2\vert .$$
By Lemma \ref{degree comparison}, $M$ is a free $H$-module, namely, $H$ has detect property in $grH$ with detecting set $\{E_1',E_2\}$. By Theorem \ref{3.2 steenrod re}, we are done.
\end{proof}
\begin{prop}\label{step 3}
	For each $n\geq 2$
	$$Res:\mathbb{T}(D_{1}(n))\longrightarrow \mathbb{T}(E(n))\times \mathbb{T}(B_{2}\cap D_{1}(n))$$ is injective.
\end{prop}

\begin{proof}
By Lemma \ref{s<t s=t}, $P^0_{n+1}$ commutes with all the $P^s_t\in D_{1}(n)$, thus $<P^0_{n+1}>$ is a normal Hopf subalgebra and $D_{1}(n)//<P^0_{n+1}>$ is a connected   Hopf algebra, denote it as $H$. Moreover, denote  $E(n)//<P^0_{n+1}>$ and $B_2\cap D_{1}(n)//<P^0_{n+1}>$ as $E_1'$ and $E_2$ respectively, then $E_1'$ and $E_2$ are connected   Hopf subalgebra of $H$, and we will show that $H$ has detect property in $grH$ with detecting set $\{E_1',E_2\}$.

Suppose $M$ is a finitely generated   $H$-module, and $M\vert _{E'_1}$, $M\vert _{E_2}$ are free. We want to show that $M$ is a free $H$-module.

By assumption $n\geq2$, thus $\vert A(1)\vert =6<\vert P^0_{n+1}\vert =2^{n+1}-1$. Therefore $A(1)\cap <P^0_{n+1}>^+D_1(n)=0$, which implies $\pi \vert _{A(1)}$ is an isomorphism of graded algebras from $A(1)$ to $\pi A(1)$, where $\pi :D_{1}(n)\to D_{1}(n)//<P^0_{n+1}>$ is the quotient map. That is, $\pi A(1) $ has detect property in $gr(\pi A(1)) $ with detecting set $\{\pi E(1)\}$. Since $M$ is free on $E_1'=\pi E(n)$ and $\pi E(1)$ is a normal Hopf subalgebra of $\pi E(n)$, $M$ is a  free    $\pi E(1)$-module by Theorem \ref{graded lagrange rewrite}. It follows that $M$ is a free  graded $\pi A(1)$-module. Now we denote $E_1=\pi A(1)$.

Consider the element  $t_1=\overline{P^0_1P^1_1}\in E_1$ and $t_2=\prod_{t=2}^{n}\overline{P^0_tP^1_t}\in E_2$. By Corollary \ref{chengji feiling} and Lemma \ref{AB^+}, they are nonzero. Next, we show they satisfy the assumption (2) of Lemma \ref{degree comparison}.

 Take $z=\overline{P^{0}_{2}}\in E_1\cap E_2$ we know $z^2=0$. Since $\vert t_1\vert =3$ and $\vert \overline{P^{1}_{2}}\vert =6$, by Lemma \ref{jisuan}, $\forall y\in E^+_2$, $\vert y\vert >\vert t_1\vert $ as long as $y\neq z$. Moreover, by Corollary \ref{chengji feiling} and \ref{AB^+}, $zt_1\neq 0$.

By definition, 
$$\vert t_1t_2\vert =\sum_{t=1}^n3\left(2^t-1\right).$$
Notice that the Milnor basis of highest degree in $D_1(n)$ is that $Sq(r_1,r_2,\cdots r_{n+1})$ with $r_t=3$ for $t\leq n$, and $r_{n+1}=1$, one has 
$$\vert D_1(n)\vert =2^{n+1}-1+\sum_{t=1}^{n}3\left(2^t-1\right).$$
Since $\vert P^0_{n+1}\vert =2^{n+1}-1$, by Theorem \ref{graded lagrange rewrite},
$$\vert H\vert =\vert D_1(n)\vert -\vert <P^0_{n+1}>\vert =\sum_{t=1}^n3\left(2^t-1\right).$$
Therefore, we have $\vert H\vert =\vert t_1\vert +\vert t_2\vert .$
By Lemma \ref{degree comparison}, $M$ is a free $H$-module, namely, $H$ has detect property in $grH$ with detecting set $\{E_1',E_2\}$. By Theorem \ref{3.2 steenrod re}, we are done.

\end{proof}

\begin{thm}
	Given $n\geq 2$, the morphism of groups
$$f:\mathbb{Z}\oplus \mathbb{Z}\to \mathbb{T}(A(n))$$
which sends $(m,l)$ to  $[\sigma(m)\Omega_{A(n)}^l(\mathbb{Z}/2)]$ is an isomorphism of groups.
\end{thm}

\begin{proof}
By Proposition \ref{step 0} \ref{step 1} \ref{step 2} \ref{step 3}, one gets $$Res: \mathbb{T}(A(n))\longrightarrow \prod_L \mathbb{T}(L) $$
is an injective group homomorphism, where each $L$ is elementary and contains $O_{a+1}$. By Proposition \ref{end of induction}, $Res^{L}_{O_{a+1}}: \mathbb{T}(L)\to \mathbb{T}(O_{a+1})$ is an isomorphism. Compose $Res$ with $Res^{L}_{O_{a+1}}$ on each component of $\prod_L\mathbb{T}(L)$, one gets an injection
$$Res: \mathbb{T}(A(n))\longrightarrow  \mathbb{T}(O_{a+1}). $$
Recall that $\sigma(m)\Omega_{A(n)}^l(\mathbb{Z}/2)\in \mathbb{T}(A(n))$ and it maps to $\sigma(m)\Omega_{O_{a+1}}^l(\mathbb{Z}/2)$ under the restriction $Res^{A(n)}_{O_{a+1}}$, the above injection is in fact an isomorphism.
\end{proof}

Recall that the group of endotrivial modules are the same as the stable Picard group, thus one has the following Corollary:
\begin{cor}\label{main result}
	Given $n\geq 2$, the morphism of groups
$$f:\mathbb{Z}\oplus \mathbb{Z}\to Pic(A(n))$$
which sends $(m,l)$ to  $[\sigma(m)\Omega_{A(n)}^l(\mathbb{Z}/2)]$ is an isomorphism of groups.
\end{cor}

\backmatter

\bibliography{sn-bibliography}

\end{document}